\documentclass[reqno]{amsart}

\title{ON GENERALIZATIONS OF BOEHMIAN SPACE AND HARTLEY TRANSFORM}
\author{C. GANESAN}
\address{Department of Mathematics, V. H. N. S. N. College, Virudhunagar - 626 001, India.}
\email{c.ganesan28@yahoo.com}

\author{R. ROOPKUMAR$^\ast$}
\address{Department of Mathematics, Central University of Tamil Nadu, Thiruvarur - 610101, India.}
\email{roopkumarr@rediffmail.com}
\thanks{$^\ast$Corresponding author}

\usepackage{amsmath,amssymb,mathrsfs}
\theoremstyle{plain}

\newtheorem{defn}{Definition}[section]
\newtheorem{thm}[defn]{Theorem}
\newtheorem{lem}[defn]{Lemma}

\newtheorem{exa}[defn]{Example}
\numberwithin{equation}{section}

\date{October, 06, 2016}

\begin{document}

\begin{abstract}
Boehmians are quotients of sequences which are constructed by using a set of axioms. In particular, one of these axioms states that the set $S$ from which the {\it denominator} sequences are formed should be a commutative semigroup with respect to a binary operation. In this paper, we introduce a generalization of abstract Boehmian space, called  generalized Boehmian space or $G$-Boehmian space,  in which $S$ is not necessarily a commutative semigroup. Next, we provide an example of a $G$-Boehmian space and we discuss an extension of the Hartley transform on it.  Finally, we compare the Hartley transform on Boehmians introduced in this paper with the existing works on Hartley transform on Boehmians. 
\end{abstract}
\maketitle

\section{Introduction}
Motivated by the Boehme's regular operators \cite{tkb}, a generalized function space called Boehmian space is introduced by J. Mikusi\'{n}ski and P. Mikusi\'{n}ski \cite{jmpm} and two notions of convergence called $\delta$-convergence and $\Delta$-convergence on a Boehmian space are introduced  in \cite{mcb}. In general, an abstract Boehmian space is constructed by using a suitable topological vector space $\Gamma$, a subset $S$ of $\Gamma$, $\star:\Gamma\times S\to \Gamma$ and a collection $\Delta$ of sequences satisfying some axioms. In \cite{mfob}, the abstract Boehmian space is generalized by replacing $S$ with a commutative semi-group in such a way that $S$ is not even comparable with $\Gamma$ and the binary operation on $S$ need not be the same as $\star$. Using this generalization of Boehmians, a lot of Boehmian spaces have been constructed for extending various integral transforms. To mention a few recent works on Boehmians, we refer to \cite{ns1,ns2,rrr,rgt,rst,ctb}. There is yet another generalization of Boehmians called generalized quotients or pseudoquotients \cite{gq1,gq2,gpq}. 

 According to the earlier constructions, we note that $S$ is assumed to be a commutative semi-group either with respect to the restriction of $\star$ or with respect to the binary operation defined on $S$. In this paper, we provide another generalization of an abstract Boehmian space, in which $S$ is not necessarily a commutative semi-group. We shall call such Boehmian space a generalized Boehmian space or simply a $G$-Boehmian space and we also provide a concrete example of a $G$-Boehmian space. At this juncture, we point out that in a recent interesting  paper on pseudoquotients \cite{kmpq}, the commutativity of $S$ is relaxed by Ore type condition, which is entirely different from the generalization discussed in this paper.

On the other hand, Hartley \cite{hart} introduced a Fourier like transform and it is later called Hartley transform. This transform founds a lot of applications in signal processing. In view of extending the domain of the Hartley transform, there are a very few works in the literature. In particular, in the context of Boehmians, we could find a few papers which are discussing the Hartley transform and unfortunately each of them is having some shortcomings.  Briefly writing, the Hartley transform is defined by employing wrong convolution theorems in \cite{lbd, alo12,alo13}, and  the domains of the Hartley transform in \cite{alo11,alo12} are not suitable. 

In this paper, we extend the Hartley transform to a suitable $G$-Boehmian space and study its properties.

\section{Preliminaries}
\subsection{Boehmians} \label{abd}
 From \cite{mcb}, we briefly recall the construction of a  Boehmian space $\mathscr{B}=\mathscr{B}(\Gamma,S, \star, \Delta)$, where $\Gamma$  is a topological vector space over $\mathbb{C}$, $S \subseteq \Gamma$,  $\star :\Gamma\times S\to \Gamma$ satisfies the following conditions:
\begin{enumerate}
\item[$(A_1)$] $(g_1+g_2)\star s= g_1\star s + g_2\star s$, $\forall g_1,g_2\in \Gamma$ and $\forall s\in S$,
\item[$(A_2)$] $(cg)\star s= c(g\star s)$, $\forall c\in \mathbb{C}$, $\forall g\in \Gamma$ and $\forall s\in S$,
\item[$(A_3)$] $g\star(s\star t) = (g\star s)\star t$, $\forall g\in \Gamma$ and $\forall s,t\in S$,
\item[$(A_4)$] $s\star t=t\star s,\ \forall s,t\in S$,
\item[$(A_c)$] If $g_n\to g$ as $n\to \infty$ in $\Gamma$ and $s\in S$, then $g_n\star s\to g\star s$ as $n\to\infty$ in $\Gamma$,
\end{enumerate}
and  $\Delta$ is a collection of sequences from $S$ with the following properties:
\begin{enumerate}
\item[$(\Delta_1)$] If $(s_n),(t_n)\in \Delta$, then $(s_n\star t_n)\in \Delta$,
\item[$(\Delta_2)$] If $g\in \Gamma$ and $(s_n)\in \Delta$, then $g\star s_n\to g$ as $n\to \infty$ in $\Gamma$.
\end{enumerate}

We call a pair $((g_n),(s_n))$  of sequences satisfying the conditions $g_n \in \Gamma,\, \forall n\in \mathbb{N}$, $(s_n)\in \Delta$ and $$g_n\star s_m=g_m\star s_n,\ \forall m,n \in \mathbb{N},$$ 
 a quotient and is denoted by $\frac{g_n}{s_n}$.  The equivalence class $\left[\frac{g_n}{s_n}\right]$ containing $\frac{g_n}{s_n}$ induced by the equivalence relation $\sim$ defined on  the collection of all quotients by 
\begin{equation} \label{rel}
\frac{g_n}{s_n} \sim \frac{h_n}{t_n} \mbox{ if } g_n\star t_m=h_m\star s_n,\ \forall m,n \in \mathbb{N}
\end{equation}
is called a Boehmian and the collection $\mathscr{B}$  of all Boehmians is a vector space with respect to the addition and scalar multiplication defined as follows.
$$\left[\frac{g_n}{s_n}\right] + \left[\frac{h_n}{t_n}\right] = \left[\frac{g_n\star t_n +h_n\star s_n}{s_n\star t_n}\right], \ c\left[\frac{g_n}{s_n}\right] = \left[\frac{cg_n}{s_n}\right].$$
Every member $g\in \Gamma$ can be uniquely identified as a member of $ \mathscr{B}$ by $\left[\frac{g\star s_n}{s_n}\right]$, where $(s_n)\in \Delta$ is arbitrary and  the operation $\star$ is also extended to $ \mathscr{B}\times S$ by $\left[\frac{g_n}{\phi_n}\right]\star t= \left[\frac{g_n\star t}{\phi_n}\right]$.
 There are two notions of convergence on $\mathscr{B}$ namely $\delta$-convergence and $\Delta$-convergence which are defined as follows.
\begin{defn}\cite[$\delta$-convergence]{mcb}
We say that $X_m\stackrel{\delta}{\to} X$ as $m\to \infty$ in $\mathscr{B}$, if there exists  $(s_n)\in \Delta$ such that $X_m \star \delta_n, X\star \delta_n \in \Gamma$, $\forall m,n\in \mathbb{N}$ and for each 
$n\in \mathbb{N}$,  $X_m \star \delta_n\to X \star \delta_n$ as $m\to \infty$ in $\Gamma$.
\end{defn}
\begin{defn}\cite[$\Delta$-convergence]{mcb}
We say that $X_m\stackrel{\Delta}{\to} X$ as $m\to \infty$ in $\mathscr{B}$, if there exists $(s_n)\in \Delta$ such that $(X_m-X) \star \delta_m \in \Gamma$, $\forall m \in \mathbb{N}$ and 
$(X_m -X) \star \delta_m\to 0$ as $m\to \infty$ in $\Gamma$.
\end{defn}

\subsection{Hartley transform}
For an arbitrary integrable function $f$, the Hartley transform  was defined by
\begin{equation}\label{dht}
[\mathcal{H}(f)](t) = \frac{1}{\sqrt{2\pi}} \int\limits_{-\infty}^{\infty} f(x) [\cos xt + \sin xt] dx,\ \forall t\in \mathbb{R}
\end{equation} and its inverse is obtained from the formula $\mathcal{H}[\mathcal{H}(f)] = f$, whenever $\mathcal{H}(f)\in \mathscr{L}^1(\mathbb{R})$. For more details on the classical theory of Hartley transform, we refer to \cite{rnb}.

The Hartley transform is one of the integral transforms which is closely related to Fourier transform in the following sense.
$$
\mathcal{F}(f)= \frac{\mathcal{H}(f)+\mathcal{H}(-f)}{2}+i \frac{\mathcal{H}(f)-\mathcal{H}(-f)}{2} \mbox{ and } 
\mathcal{H}(f)= \frac{1+i}{2}\mathcal{F}(f)+i\frac{1-i}{2}\mathcal{F}(-f),
$$
where $\mathcal{F}(f)$ is the Fourier transform of $f$, which is defined by $$\mathcal{F}(f)(t) = \frac{1}{\sqrt{2\pi}} \int\limits_{-\infty}^{\infty} f(x) e^{-ixt} dx,\ \forall t\in \mathbb{R}.$$
However N. Sundararajan \cite{ns} pointed out that Hartley transform has some computational advantages over the Fourier transform and therefore it can be an ideal alternative of Fourier transform. 

Furthermore, as $|[\mathcal{H}(f)](t)|\le 2|\mathcal{F}(f)(t)|$, $\forall t\in \mathbb{R}$, using the properties of Fourier transform, we have $\mathcal{H}(f)\in C_0(\mathbb{R})$, $\|\mathcal{H}(f)\|_\infty\le 2\|\mathcal{F}(f)\|_\infty\le \|f\|_1$ and hence the Hartley transform $\mathcal{H}: \mathscr{L}^1(\mathbb{R})\to C_0(\mathbb{R})$ is continuous.  

\section{Generalized Boehmian spaces}
We introduce a generalization of  Boehmain space called $G$-Boehmian space $\mathscr{B}^\star(\Gamma,S, \star, \Delta)$,  which is obtained by relaxing the Boehmian-axiom $(A_4)$ in Subsection \ref{abd} by $$(A_4^\prime)\ f\star(s\star t)=(f\star t)\star s, \ \forall f\in \Gamma\mbox{ and }s,t\in S.$$ If we probe into know the necessity for introducing the axioms $(A_3)$ and $(A_4)$ for constructing Boehmians, we could see that these two axioms are used to prove the transitivity of the relation $\sim$ defined on the collection of all quotients in \eqref{rel}.

It is easy to see that the verification of reflexivity and symmetry for the relation $\sim$ are straightforward. So we now verify the transitivity of $\sim$ using $(A_3)$ and $(A_4^\prime)$. 

Let $\frac{g_n}{s_n}$, $\frac{h_n}{t_n}$ and $\frac{p_n}{u_n}$ be quotients such that 
$\frac{g_n}{s_n}\sim \frac{h_n}{t_n}$ and $\frac{h_n}{t_n}\sim \frac{p_n}{u_n}$. Then, we have $g_n, h_n, p_n\in \Gamma$, $\forall n\in \mathbb{N}$, $(s_n)$, $(t_n)$, $(u_n) \in \Delta$ and 
\begin{eqnarray}
g_n\star s_m&=& g_m\star s_n,\ \forall m,n\in \mathbb{N} \label{q1}\\
h_n\star t_m&=& h_m\star t_n,\ \forall m,n\in \mathbb{N} \label{q2}\\
p_n\star u_m&=& p_m\star u_n,\ \forall m,n\in \mathbb{N} \label{q3}\\
g_n\star t_m&=& h_m\star s_n,\ \forall m,n\in \mathbb{N} \label{e1}\\
h_n\star u_m&=& p_m\star t_n,\ \forall m,n\in \mathbb{N}.  \label{e2}
\end{eqnarray}
For arbitrary $m,n,j\in \mathbb{N}$, applying $(A_4^\prime)$, $(A_3)$, \eqref{q1}, \eqref{q2}, \eqref{q3}, \eqref{e1}, and \eqref{e2} we get 
$$
\begin{array}{lllll}
(g_n\star u_m)\star t_j&=& g_n\star (t_j\star u_m)
&=& (g_n\star t_j)\star u_m\\
&=& (h_j\star s_n)\star u_m
&=&h_j\star (u_m\star s_n)\\
&=&(h_j\star u_m)\star s_n
&=&(p_m\star t_j)\star s_n\\
&=&p_m\star (s_n\star t_j)
&=&(p_m\star s_n)\star t_j.
\end{array}
$$
Next applying $(\Delta_2)$, we get $g_n\star u_m=p_m\star s_n$, $\forall m,n\in \mathbb{N}$, and hence $\frac{g_n}{s_n}\sim \frac{p_n}{u_n}$. Thus, the transitivity of $\sim $ follows.

We note that the axioms $(A_3)$ and $(A_4)$ are also used in the proof of the following statements:
\begin{itemize}
\item $\frac{g\star s_n}{s_n}$ is a quotient, $\forall g\in \Gamma$ and $(s_n)\in \Delta$,
\item $\frac{g_n}{s_n}\sim \frac{g_n\star t_n}{s_n\star t_n}$, for each quotient $\frac{g_n}{s_n}$ and for each $(t_n)\in \Delta$,
\item $\frac{g_n\star t}{s_n}$ is a quotient whenever $\frac{g_n}{s_n}$ is a quotient,
\item $\frac{g_n\star t_n+h_n\star s_n}{s_n\star t_n}$ is a quotient whenever $\frac{g_n}{s_n}$ and $\frac{h_n}{t_n}$ are quotients,
\end{itemize}
and these statements can also be proved by using $(A_3)$ and $(A_4^\prime)$ as above.

Now we construct an example of a $G$-Boehmian space by proving the required auxiliary results. Let $\Gamma=S=\mathscr{L}^1(\mathbb{R})$, $\Delta$ be the usual collection of all sequences $(\delta_n)$ from $\mathscr{L}^1(\mathbb{R})$ satisfying the following properties.
\begin{enumerate}
\item[$(P_1)$] $\int\limits_{-\infty}^\infty \delta_n(t)\, dt=1,\ \forall n\in \mathbb{N}$,
\item[$(P_2)$] $\int\limits_{-\infty}^\infty |\delta_n(t)|\, dt\le M,\ \forall n\in \mathbb{N}$, for some $M>0$,
\item[$(P_3)$] supp $\delta_n \to 0$ as $n\to \infty$, where supp $\delta_n$ is the support of $\delta_n$;
\end{enumerate}
and $\#$ be the following convolution
\begin{equation}
(f \#  g)(x) = \frac{1}{2} \int\limits_{-\infty}^\infty [f(x+y) + f(x-y)]g(y) dy,\ \forall x\in \mathbb{R},
\end{equation}
for all $f, g\in \mathscr{L}^1(\mathbb{R})$.
\begin{lem}
If $f,g\in \mathscr{L}^{1}(\mathbb{R})$, then $\|f  \#    g\|_1 \le \|f\|_1 \|g\|_1$ and hence $f\# g\in \mathscr{L}^{1}(\mathbb{R})$.\label{fce}
\end{lem}
\begin{proof}
By using Fubini's theorem, we obtain \\
$\|f  \#    g\|_1 $
$\begin{array}[t]{lll}
&=& \frac{1}{2} \int\limits_{-\infty}^{\infty}\left| \int\limits_{-\infty}^{\infty} [f(x+y) + f(x-y)]g(y) dy \right| dx\\
& \leq& \frac{1}{2} \int\limits_{-\infty}^{\infty} \int\limits_{-\infty}^{\infty} |[f(x+y) + f(x-y)] g(y)| dy \, dx\\
&\leq& \frac{1}{2} \int\limits_{-\infty}^{\infty}|g(y)| \int\limits_{-\infty}^{\infty} |f(x+y) + f(x-y)|  dx \, dy\\
&\leq&  \|f\|_1 \|g\|_1<+\infty
\end{array}$\\[.1cm]
 and hence   $f  \#   g \in \mathscr{L}^1(\mathbb{R})$. 
\end{proof}
\begin{lem}
If $f, g$ and $h \in L^1(\mathbb{R})$ then $(f \# g) \# h = f\# (g \# h) = (f\# h) \# g.$
\end{lem}
\begin{proof}
Let $f, g, h \in L^1(\mathbb{R})$ and let $x\in\mathbb{R}$. Repeatedly applying the Fubini's theorem, we get that\\
$[f\# (g \# h)](x)$
\begin{eqnarray}
&=& \int\limits_{-\infty}^{\infty} [f(x+y) + f(x-y)] (g\# h)(y) dy \nonumber\\
&=&\int\limits_{-\infty}^{\infty} [f(x+y) + f(x-y)] \int\limits_{-\infty}^{\infty} [g(y+z) + g(y-z)] h(z) dz dy \nonumber\\
&=&\int\limits_{-\infty}^{\infty} h(z) \left(\int\limits_{-\infty}^{\infty} [f(x+y) + f(x-y)] g(y+z) dy \right.  \nonumber\\
&& \left. + \int\limits_{-\infty}^{\infty} [f(x+y) + f(x-y)]g(y-z) dy \right)\, dz \nonumber\\
&=&\int\limits_{-\infty}^{\infty} h(z) \left(\int\limits_{-\infty}^{\infty} [f(x+u-z) + f(x-u+z)] g(u) du \right. \nonumber\\
&&\left.  + \int\limits_{-\infty}^{\infty} [f(x+u+z) + f(x-u-z)]g(u) du\right) dz, \nonumber\\
& &(\mbox{by using $y+z = u$ in the first term and $y-z = u$ in the second term}) \nonumber\\
&=&\int\limits_{-\infty}^{\infty} h(z) \int\limits_{-\infty}^{\infty} [f(x+u-z) + f(x-u+z)   \nonumber\\
&& +f(x+u+z) + f(x-u-z)]g(u) du\, dz \nonumber\\
&=&\int\limits_{-\infty}^{\infty} h(z) \left(\int\limits_{-\infty}^{\infty} [f(x+z+u) +f(x+z-u)] g(u) du \right.  \nonumber\\
&&\left. + \int\limits_{-\infty}^{\infty}[f(x-z+u) +  f(x-z-u)]g(u) du\right) dz \label{asso}
\end{eqnarray}
\begin{eqnarray}
 &=& \int\limits_{-\infty}^{\infty} h(z) [(f\# g) (x+z) + (f\# g)(x-z)] dx \nonumber\\
 & = &[(f\# g) \# h](x). \nonumber
\end{eqnarray}

Using \eqref{asso}, we get 
\begin{eqnarray*}
[f\# (g \# h)](x)&=&\int\limits_{-\infty}^{\infty} h(z) \int\limits_{-\infty}^{\infty} [f(x+z+u) +f(x+z-u)\\
&& + f(x-z+u) +  f(x-z-u)]g(u) du\, dz \\
&=&\int\limits_{-\infty}^{\infty} g(u) \int\limits_{-\infty}^{\infty}[f(x+z+u) + f(x+z-u)  \\
&&+f(x-z+u) + f(x-z-u)] h(z)dz\, du\\
&=&\int\limits_{-\infty}^{\infty} g(u) \int\limits_{-\infty}^{\infty}[f(x+u+z) + f(x+u-z) +f(x-u+z) \\
&&+ f(x-u-z)] h(z)dz\, du\\
&=&\int\limits_{-\infty}^{\infty} g(u) \left[\int\limits_{-\infty}^{\infty}[f(x+u+z) + f(x+u-z)]h(z) dz \right.  \\
&&\left.+\int\limits_{-\infty}^{\infty}[f(x-u+z)+ f(x-u-z)] h(z)dz\right] du\\
&=&\int\limits_{-\infty}^{\infty} g(u) [(f\# h)(x+u) + (f\# h)(x-u)] du \\
&=& [(f\# h) \# g](x).
\end{eqnarray*}
Since $x\in \mathbb{R}$ is arbitrary, the proof follows.
\end{proof}

\begin{lem}
If $f_n \to f$ as $n\to \infty$ in $\mathscr{L}^1(\mathbb{R})$ and if $g\in \mathscr{L}^1(\mathbb{R})$, then $f_n  \#   g \to f \#   g$ as $n\to \infty$ in $\mathscr{L}^1(\mathbb{R}).$
\end{lem}
\begin{proof} 
From the proof of Lemma $\ref{fce}$, we have the estimate  
\begin{equation}
\|(f_n-f)\# g\|_1\le \|f_n-f\|_1 \| g\|_1. \label{fng}
\end{equation}
Since $f_n\to f$ as $n\to \infty$ in $\mathscr{L}^1(\mathbb{R})$, the right hand side of (\ref{fng}) tends to zero as $n\to \infty$. Hence the lemma follows.
\end{proof}

\begin{lem}
If $(\delta_n),\, (\psi_n) \in \Delta$ then $(\delta_n \# \psi_n) \in \Delta$. \label{deltanstarpsin}
\end{lem}
\begin{proof}
By using Fubini's theorem, we get\\
$\begin{array}{lll}
\int\limits_{-\infty}^\infty (\delta_n \#  \psi_n)(x) \,dx& =& \frac{1}{2}\int\limits_{-\infty}^\infty \int\limits_{-\infty}^\infty [\delta_n(x+y) + \delta_n(x-y)]\,\psi_n(y)\,dy\,dx\\
&=&\frac{1}{2}\int\limits_{-\infty}^\infty \psi_n(y)\int\limits_{-\infty}^\infty [\delta_n(x+y) + \delta_n(x-y)]\, dx\, dy\\
&=&\frac{1}{2}\int\limits_{-\infty}^\infty \psi_n(y)\left[\int\limits_{-\infty}^\infty \delta_n(z)\, dz + \int\limits_{-\infty}^\infty \delta_n(z)\, dz\right] dy\\
&=&\frac{1}{2} \int\limits_{-\infty}^\infty 2 \psi_n(y) dy  = 1, \mbox{ for all $n \in \mathbb{N}.$}\\
\end{array}$\\
By a similar argument, it is easy to verify that $\int\limits_{-\infty}^\infty |(\delta_n \#  \psi_n)(x)| \,dx \leq M$ for some $M >0$.  Since $\mbox{supp } \delta_n \#  \psi_n \subset [\mbox{supp }\delta_n + \mbox{supp }\psi_n] \cup [\mbox{supp }\delta_n - \mbox{supp }\psi_n]$, we get that
supp $(\delta_n\#  \psi_n)\to \{0\}$ as $n\to\infty$. Hence it follows that $(\delta_n \#  \psi_n)\in \Delta$.
\end{proof}

\begin{thm}
Let $f\in \mathscr{L}^1(\mathbb{R})$ and let $(\delta_n)\in \Delta$, then $f \#  \delta_n \to f$ as $n\to \infty$ in $\mathscr{L}^1(\mathbb{R})$. \label{f*del}
\end{thm}
\begin{proof}
Let $\epsilon > 0$ be given. By the property $(P_2)$ of $(\delta_n)$, there exists $M>0$ with $\int\limits_{-\infty}^\infty |\delta_n(t) | dt \leq M, \, \forall\ n\in \mathbb{N}$.  Using the  continuity of the mapping $y \mapsto f_y$ from $\mathbb{R}$ in to $\mathscr{L}^1(\mathbb{R})$, (see \cite[Theorem 9.5]{rcrudin}), choose $\delta > 0$ such that 
\begin{equation}
\| f_y - f_0 \|_1 < \frac{\epsilon}{M}\mbox{ whenever } |y| < \delta, \label{fy}
\end{equation}
where $f_y(x)=f(x-y)$, $\forall x \in \mathbb{R}$. By the property $(P_3)$ of $(\delta_n)$ there exists $N\in \mathbb{N}$ with supp $\delta_n \subset [-\delta, \delta], \ \forall\, n\geq N$. By using the property $(P_1)$ of $(\delta_n)$ and the Fubini's theorem, we obtain\\
$\| f\#  \delta_n - f\|_1 $\\
$\begin{array}[t]{lll}
&=& \int\limits_{-\infty}^{\infty} \left|\frac{1}{2}\int\limits_{-\infty}^{\infty} [f(x+y) + f(x-y)]\ \delta_n(y) dy  - f(x) \int\limits_{-\infty}^{\infty}\delta_n(y) dy\right| dx\\
&\leq& \frac{1}{2}\int\limits_{-\infty}^{\infty} \int\limits_{-\infty}^{\infty} \left(|f(x+y) - f(x)|+ |f(x-y) - f(x)|\right) \ |\delta_n(y)| dx  dy\\
&\leq& \frac{1}{2}\int\limits_{-\infty}^{\infty}\left(\int\limits_{-\infty}^{\infty} |f(x+y) - f(x)| dx  + \int\limits_{-\infty}^{\infty} |f(x-y) - f(x)|dx\right) \ |\delta_n(y)|  dy
\end{array}$\\
$\begin{array}[t]{lll}
&=&\frac{1}{2}\int\limits_{-\delta}^{\delta} (\|f_{-y} - f_0\|_1 + \|f_y - f_0\|_1) \ |\delta_n(y)| dy, \ \forall\ n\geq N\\
&<&\frac{1}{2}\int\limits_{-\delta}^{\delta} (\frac{\epsilon}{M} + \frac{\epsilon}{M}) \ |\delta_n(y)| dy, \ \mbox{ by (\ref{fy})}\\
&=&\frac{\epsilon}{M} \int\limits_{-\delta}^{\delta}\ |\delta_n(y)| dy \leq \epsilon,\ \forall \ n\geq N
\end{array}$\\
and hence $f \#  \delta_n \to f$ as $n\to \infty$ in $\mathscr{L}^1(\mathbb{R})$.
\end{proof}
\begin{lem}
If $f_n \to f$ as $n\to \infty$ in $\mathscr{L}^1(\mathbb{R})$ and $(\delta_n) \in \Delta$, then $f_n \#  \delta_n \to f$ as $n\to \infty$ in $\mathscr{L}^1(\mathbb{R})$.
\end{lem}
\begin{proof}
For any $n\in \mathbb{N}$ we have,
\begin{eqnarray*}
\|f_n \#  \delta_n - f \|_1 &=& \|f_n \#  \delta_n -f \#  \delta_n + f \#  \delta_n- f\|_1\\
&\le& \|(f_n-f) \#  \delta_n\|_1 + \| f \#  \delta_n - f\|_1\\
&\le&\|f_n - f\|_1 \ \|\delta_n\|_1 + \| f \#  \delta_n - f\|_1, \mbox{ (by Lemma \ref{fce})}\\
&\le& M \|f_n - f\|_1  + \| f \#  \delta_n - f\|_1
\end{eqnarray*}
Since $f_n\to f$ as $n\to \infty$ in $\mathscr{L}^1(\mathbb{R})$ and by Theorem $\ref{f*del}$, the right hand side of the last inequality tends to zero as $n\to \infty$. Hence the lemma follows.
\end{proof}
Thus the $G$-Boehmian space $ \mathscr{B}^\star_{\mathscr{L}^1}=\mathscr{B}^\star(\mathscr{L}^1(\mathbb{R}),\mathscr{L}^1(\mathbb{R}),\#,\Delta)$ has been constructed. 

Finally, we justify that the convolution $\#$ introduced in this section is not commutative.

\begin{exa}
If $f(x) = \begin{cases}
e^{-x} &\text{ if } x \geq 0\\
0 &\text{ if } x < 0
\end{cases}$  and $g(x) = \begin{cases}
0 &\text{ if } x > 0\\
e^{x} &\text{ if } x \leq 0,
\end{cases}$ then $f, g \in L^1(\mathbb{R})$ and $f\# g \neq g\# f$. 
\end{exa}
Indeed, for any $x\in \mathbb{R}$, we have\\
$(f\# g)(x) = \int\limits_{-\infty}^{\infty} [f(x+y) + f(x-y)] g(y) \, dy = \int\limits_{-\infty}^{0} [f(x+y) + f(x-y)] e^y dy $\\
$\begin{array}[t]{lll}
&=& \int\limits_{-\infty}^{0} f(x+y) e^y dy +\int\limits_{-\infty}^{0}  f(x-y) e^y dy \\
&=&\begin{cases}
\int\limits_{-x}^{0} e^{-(x+y)} e^y dy +\int\limits_{-\infty}^{0}  e^{-(x-y)} e^y dy &\text{ if } x\geq 0 \\
 0 + \int\limits_{-\infty}^{x}  e^{-(x-y)} e^y dy &\text{ if } x< 0 
\end{cases}\\
&=&\begin{cases}
e^{-x} \left(\int\limits_{-x}^{0}  dy +\int\limits_{-\infty}^{0}  e^{2y} dy\right) &\text{ if } x\geq 0 \\
e^{-x} \int\limits_{-\infty}^{x}   e^{2y} dy &\text{ if } x< 0 
\end{cases} \\
&=& \begin{cases}
e^{-x} (x+\frac{1}{2}) &\text{ if } x\geq 0 \\
\frac{e^{x}}{2}  &\text{ if } x< 0 
\end{cases}
\end{array}$\\
 and \\
$(g\# f)(x) = \int\limits_{-\infty}^{\infty} [g(x+y) + g(x-y)] f(y) \, dy = \int\limits_{0}^{\infty} [g(x+y) + g(x-y)] e^{-y} dy  $\\
$\begin{array}[t]{lll}
&=& \int\limits_{0}^{\infty} g(x+y) e^{-y} dy +\int\limits_{0}^{\infty}  g(x-y) e^{-y} dy \\
&=&\begin{cases}
0+\int\limits_x^\infty e^{x-y} e^{-y} dy &\text{ if } x > 0\\
\int\limits_0^{-x} e^{x+y} e^{-y} dy + \int\limits_0^\infty e^{x-y} e^{-y} dy &\text{ if } x \leq 0 
\end{cases}\\
&=&\begin{cases}
e^x\int\limits_x^\infty e^{-2y} dy &\text{ if } x > 0\\
e^x\left(\int\limits_0^{-x}  dy + \int\limits_0^\infty e^{-2y} dy\right) &\text{ if } x \leq 0 
\end{cases}\\
&=&\begin{cases}
\frac{1}{2}e^{-x} &\text{ if } x > 0\\
e^x(-x+\frac{1}{2}) &\text{ if } x \leq 0. 
\end{cases}
\end{array}$

From the above computations it is clear that $f\# g \neq g\# f$ and hence our claim holds.

\section{Hartley transform on $G$-Boehmians}
As in the general case of extending any integral transform to the context of Boehmians, we have to first obtain a suitable convolution theorem for Hartley transform. To obtain a compact version of a convolution theorem for Hartley transform, for $f\in \mathscr{L}^1(\mathbb{R})$, we define 
\begin{equation}
[\mathcal{C}(f)](t)=\int\limits_{-\infty}^\infty f(x)\cos xt \, dx, \ t\in \mathbb{R}.
\end{equation}
We point out that $\mathcal{C}$ is not the usual Fourier cosine transform, as Fourier cosine transform is defined for integrable functions on non-negative real numbers.
\begin{thm}
 If $f,g\in \mathscr{L}^{1}(\mathbb{R})$, then $\mathcal{H}(f \#  g) = \mathcal{H}(f)\, \cdot \mathcal{C}(g).$\label{conH}
\end{thm}
\begin{proof}
Let $t\in \mathbb{R}$ be arbitrary. By using Fubini's theorem, we obtain that\\
$[\mathcal{H}(f\#  g)](t) = \frac{1}{\sqrt{2\pi}}\int\limits_{-\infty}^{\infty} (f\#  g)(x) [\cos xt + \sin xt ] dx$\\
$\begin{array}[t]{lll}
&=&\frac{1}{\sqrt{2\pi}}\frac{1}{2} \int\limits_{-\infty}^{\infty} \int\limits_{-\infty}^\infty [f(x+y) + f(x-y)]g(y) dy [\cos xt + \sin xt ] dx\\
&=&\frac{1}{2\sqrt{2\pi}} \int\limits_{-\infty}^{\infty} g(y) \int\limits_{-\infty}^\infty [f(x+y) + f(x-y)]\ [\cos xt + \sin xt ]dx \, dy\\
&=&\frac{1}{2\sqrt{2\pi}} \int\limits_{-\infty}^{\infty} g(y)  \left(\int\limits_{-\infty}^\infty f(x+y)\cos xt dx +\int\limits_{-\infty}^\infty f(x+y)\sin xt dx \right.\\
& &\left. + \int\limits_{-\infty}^\infty f(x-y)\cos xt dx  + \int\limits_{-\infty}^\infty f(x-y)\sin xt dx \right) dy\\
&=& \frac{1}{2\sqrt{2\pi}} \int\limits_{-\infty}^{\infty} g(y)  \left(\int\limits_{-\infty}^\infty f(z) \cos (zt-yt) dz + \int\limits_{-\infty}^\infty f(z) \sin (zt-yt) dz \right.\\
& &\left. + \int\limits_{-\infty}^\infty f(z) \cos (zt+yt) dz + \int\limits_{-\infty}^\infty f(z) \sin (zt+yt) dz \right) dy
\end{array}$\\
$\begin{array}[t]{lll}
&=&\frac{1}{2\sqrt{2\pi}} \int\limits_{-\infty}^{\infty} g(y)  \int\limits_{-\infty}^\infty f(z) [2 \cos zt \cos yt  + 2\sin zt \cos yt ]dz\, dy\\
&=&\frac{1}{\sqrt{2\pi}} \int\limits_{-\infty}^{\infty} g(y) \cos yt \int\limits_{-\infty}^\infty f(z) [ \cos zt   + \sin zt ]dz\, dy\\
&=& [\mathcal{H}(f)](t)\cdot [\mathcal{C}(g)](t).
\end{array}$\\
Thus we have $\mathcal{H}(f \#  g) = \mathcal{H}(f)\, \cdot \mathcal{C}(g).$
\end{proof}

\begin{thm}
If $f,g\in \mathscr{L}^{1}(\mathbb{R})$, then $\mathcal{C}(f \#  g) = \mathcal{C}(f) \cdot \mathcal{C}(g).$\label{conC}
\end{thm}
\begin{proof}
Let $t\in \mathbb{R}$ be arbitrary. By using Fubini's theorem, we obtained\\
$[\mathcal{C}(f\#  g)](t) = \int\limits_{-\infty}^{\infty} (f\#  g)(x) \cos xt\, dx$\\
$\begin{array}[t]{lll}
&=&\frac{1}{2} \int\limits_{-\infty}^{\infty} \int\limits_{-\infty}^\infty [f(x+y) + f(x-y)]g(y) dy\, \cos xt  dx\\
&=&\frac{1}{2} \int\limits_{-\infty}^{\infty} g(y)  \left(\int\limits_{-\infty}^\infty f(x+y)\cos xt dx  + \int\limits_{-\infty}^\infty f(x-y)\cos xt dx \right) dy\\
&=& \frac{1}{2} \int\limits_{-\infty}^{\infty} g(y)  \left(\int\limits_{-\infty}^\infty f(z) \cos (zt-yt) dz + \int\limits_{-\infty}^\infty f(z) \cos (zt+yt) dz \right)dy\\
&=& \int\limits_{-\infty}^{\infty} g(y) \cos yt  \, \int\limits_{-\infty}^\infty f(z) \cos zt dz\, dy\\
&=&  [\mathcal{C}(f)](t)\cdot [\mathcal{C}(g)](t).\\
\end{array}$\\
Since $t\in \mathbb{R}$ is arbitrary, we have $\mathcal{C}(f \#  g) = \mathcal{C}(f) \cdot \mathcal{C}(g).$
\end{proof}

\begin{lem}
If $(\delta_n) \in \Delta$ then $\mathcal{C}(\delta_n) \to 1$ as $n\to \infty$ uniformly on compact subset of $\mathbb{R}$. \label{tendsto1}
\end{lem}
\begin{proof}
Let $K$ be a compact subset of $\mathbb{R}$. Let $\epsilon >0$ be given. Choose $M_1>0$, $M_2>0$ and a positive integer $N$ such that  $\int_{-\infty}^\infty |\delta_n(t)|\, dt \le M_1, \, \forall n\in \mathbb{N}$, $K\subset [-M_2,M_2]$  and $supp\,\delta_n \subset [-\epsilon,\epsilon]$ for all $n \geq N$.  Then for $t\in K$ and $n\ge N$, we have \\
$\begin{array}{lll}
|[\mathcal{C}(\delta_n)](t) - 1| &=&|\int\limits_{-\infty}^\infty \delta_n(s) \cos ts \, ds - \int\limits_{-\infty}^\infty \delta_n(s)\, ds|\\
& \leq &\int\limits_{-\infty}^\infty |\delta_n(s)| \, |\cos ts - 1|\, ds\\
&=& \int\limits_{-\epsilon}^\epsilon |\delta_n(s)| \, |\cos ts - 1|\, ds,\ \forall\, n \geq N\\
&\leq& \int\limits_{-\epsilon}^\epsilon |\delta_n(s)| \,|ts|\, ds, \\
&&\mbox{(by using mean-value theorem, and $|\sin x|\le 1, \ \forall x\in \mathbb{R}$)}\\
&\leq& M_2\epsilon \int\limits_{-\epsilon}^\epsilon |\delta_n(s)| \,ds \\
&\le& M_2 M_1\epsilon.
\end{array}$\\
This completes the proof.
\end{proof}

\begin{defn}
For $\beta = \left[\frac{f_n}{\delta_n}\right]\in \mathscr{B}^\star_{\mathscr{L}^1}$, we define the extended Hartley transform of $\beta$ by $[\mathscr{H}(\beta)](t) = \lim\limits_{n\to \infty} [\mathcal{H}(f_n)](t), \ \ (t\in \mathbb{R}).$
\end{defn}

The above limit exists and is independent of the representative $\frac{f_n}{\delta_n}$ of $\beta$. Indeed, for $t\in \mathbb{R}$, choose $k$ such that $[\mathcal{C}(\delta_k)](t)\neq 0$. Then, applying Theorem \ref{conH}, we obtain that $[\mathcal{H}(f_n)](t) = \frac{[\mathcal{H}(f_n\#  \delta_k)](t)}{[\mathcal{C}(\delta_k)](t)} = \frac{[\mathcal{H}(f_k\#  \delta_n)](t)}{[\mathcal{C}(\delta_k)](t)} = \frac{[\mathcal{H}(f_k)](t)}{[\mathcal{C}(\delta_k)](t)}  [\mathcal{C}(\delta_n)](t).$ Therefore, using Lemma \ref{tendsto1}, we get  $[\mathcal{H}(f_n)](t) \to   \frac{[\mathcal{H}(f_k)](t)}{[\mathcal{C}(\delta_k)](t)}$, as $n\to \infty$ uniformly on each compact subset of $\mathbb{R}$. If $\frac{f_n}{\delta_n}\sim \frac{g_n}{\psi_n}$, then
$f_n \#  \psi_m = g_m \#  \delta_n$ for all $m,n \in \mathbb{N}$. Again using Theorem \ref{conH}, we get  $\lim\limits_{n\to\infty}[\mathcal{H}(f_n)](t) =  \frac{[\mathcal{H}(f_k)](t)}{[\mathcal{C}(\delta_k)](t)} = \frac{[\mathcal{H}(g_k)](t)}{[\mathcal{C}(\psi_k)](t)} = \lim\limits_{n\to\infty}[\mathcal{H}(g_n)](t).$

If $f\in \mathscr{L}^1(\mathbb{R})$ and  $\beta = \left[\frac{f\#  \delta_n}{\delta_n}\right]$, then $$[\mathscr{H}(\beta)](t) = \lim\limits_{n\to\infty} [\mathcal{H}(f\#  \delta_n)](t) = [\mathcal{H}(f)](t)\lim\limits_{n\to\infty} [\mathcal{C}(\delta_n)](t) = [\mathcal{H}(f)](t),$$ as $[\mathcal{C}(\delta_n)](t) \to 1$ as $n\to \infty$ uniformly on each compact subset of $\mathbb{R}$. 
This shows that the extended Hartley transform is consistent with the Hartley transform on $\mathscr{L}^1(\mathbb{R}).$
\begin{thm}
If $\beta \in \mathscr{B}^\star_{\mathscr{L}^1}$, then the extended Hartley transform $\mathscr{H}(\beta)\in C(\mathbb{R}).$
\end{thm}
\begin{proof} As $\mathscr{H}(\beta)$ is the uniform limit of $\{H(f_n)\}$ on each compact subset of $\mathbb{R}$ and each $H(f_n)$ is a continuous function on $\mathbb{R}$, $\mathscr{H}(\beta)$ is a continuous function on $\mathbb{R}$. 
\end{proof}
As proving the following properties of the  Hartley transform on Boehmians is a routine exercise, as in the case of Fourier transform on integrable Boehmians \cite{mfib}, we just state them without proofs.
\begin{thm}
The Hartley transform $\mathscr{H}:\mathscr{B}^\star_{\mathscr{L}^1}\to C(\mathbb{R})$ is linear. 
\end{thm}
\begin{thm}
The Hartley transform $\mathscr{H}:\mathscr{B}^\star_{\mathscr{L}^1}\to C(\mathbb{R})$ is one-to-one. 
\end{thm}
\begin{thm}
The Hartley transform $\mathscr{H}:\mathscr{B}^\star_{\mathscr{L}^1}\to C(\mathbb{R})$ is continuous with respect to $\delta$-convergence and $\Delta$-convergence. 
\end{thm}

\section{Concluding remarks}
As a highlight of this work, we mention that the notion of Boehmian space is generalized by $G$-Boehmian space, and an example of a $G$-Boehmian space $\mathscr{B}^\star_{\mathscr{L}^1}$  which is not a Boehmian space is constructed.  

As there are a few works on Hartley transform (and generalized Hartley transform) on Boehmian spaces \cite{lbd,alo11, alo13, alo12}, it is necessary to compare the present work with the existing works on Hartley transform for Boehmians. However, each of these papers is having some major shortcomings, and hence the present work could not be comparable with the existing works. The following is the brief information about the papers \cite{lbd,alo11, alo13, alo12}.

According to \cite{lbd}, the Hartley transform of an integrable Boehmian $[f_n/\delta_n]$ is the $\lim\limits_{n\to \infty} H(f_n)$. While proving the existence of this limit (in \cite[Lemma 1]{lbd}), the identity $H(f_n\ast \delta_k)= H(f_n)H(\delta_k)$ was used, which is not valid, as $\ast$ is the usual convolution defined by $(f\ast g)(x)=\int\limits_{\mathbb{R}} f(x-y)g(y)\, dy$.

The Hartley transform defined in \cite{alo11} is also not correct for the following reason. The $C^\infty$-Boehmian space $\mathscr{B}(C^\infty (\mathbb{R}^n), (\mathscr{D}(\mathbb{R}^n), \ast), \ast, \Delta_0)$ introduced in \cite{mcb} was denoted by  $\mathbb{M}(\mathbb{E};\mathbb{D};\Delta;\bullet)$ in \cite{alo11}.  As per \cite[Equation (14)]{alo11}, the Hartley transform of $\left[\frac{f_n}{\gamma_n}\right]\in \mathbb{M}(\mathbb{E};\mathbb{D};\Delta;\bullet)$ was defined by $H\left[\frac{f_n}{\gamma_n}\right]=\left[\frac{Hf_n}{H\gamma_n}\right]$. The blunder in this definition is that as $f_n\in \mathbb{E}=$ the space of all infinitely differentiable functions on $\mathbb{R}$, as the Hartley transform is not defined for all functions in $\mathbb{E}$, there is no justification for the existence of $Hf_n$, in the numerator of the right hand side.

In \cite{alo13}, a unified generalization of the Fourier and Hartley transforms is given by
$$F_a^b(f)(\xi)=\frac{1}{\sqrt{2\pi}} \int\limits_{\mathbb{R}} f(y) (a\cos(y\xi)+b\sin(y\xi))\, dy,\ \forall \xi\in \mathbb{R}.$$
In \cite{alo13}, the Hartley transform was extended to compactly supported distributions and then it was attempted to extend on compactly supported Boehmians. Unfortunately, the proof of the convolution theorem for the Hartley transform in \cite[Theorem 6]{alo13} is not correct for the following reasons. Let us first recall the definition of a convolution $\vee$ defined in \cite[Equation 20]{alo13}.
\begin{eqnarray}
f\vee g&=&fg+\hat{F}_a^b(f_1) F_c(g_1)\\
\label{ct2} 
&&\mbox{ where } \hat{F}_a^b(f_1)=f, F_c(g_1)=g \mbox{ for some }f_1\in E^\prime, \ g_1\in D. \nonumber
\end{eqnarray}

There are many conceptual errors in the discussions on the properties of this convolution in \cite{alo13}. First of all, using \cite[Theorem 5]{alo13} in the proof of  \cite[Theorem 6]{alo13} is not possible, as in \cite[Theorem 6]{alo13} $f$, $g$ are assumed to be compactly supported distribution and  compactly supported function, respectively,  whereas both are assumed to be integrable functions in \cite[Theorem 5]{alo13}. Moreover, the following one line proof of \cite[Theorem 6]{alo13}
$$
f\vee g=fg+\hat{F}_{-a}^b(f_1) F_s(g_1)=\hat{F}_a^b(f_1) F_c(g_1)+\hat{F}_{-a}^b(f_1) F_s(g_1)=\hat{F}_a^b(f\ast g)
$$
is valid only after 
\begin{enumerate}
\item identifying that ${F}_a^b$ and $\hat{F}_a^b$ are same,
\item correcting the definition of $f\vee g$ as $fg+\hat{F}_{-a}^b(f_1) F_s(g_1)$,
\item replacing $\hat{F}_a^b(f\ast g)$ by $\hat{F}_a^b(f_1\ast g_1)$.
\end{enumerate} 
The major mistakes in \cite{alo13} are: 
\begin{itemize}
\item In the proof of \cite[Theorem 7]{alo13},  $\hat{F}_a^b(f_1\ast(g_1\ast h_1))=f\vee \hat{F}_a^b(g_1\ast h_1)$ is not correct since as per the convolution theorem, the right hand side should be $f\vee F_c(g_1\ast h_1)$.
\item \cite[Theorem 9]{alo13} is as follows. $``${\em Let $g_1,h_1\in D$ with $F_c(g_1)=g, F_c(h_1)=h\in G_a^b$, then $g\vee h=h\vee g$.}" While $\vee$ is defined as a function on $G_a^b\times G_c$ (see \cite[Equation 20]{alo13}), proving $g\vee h=h\vee g$, for $g,h\in G_a^b$  in \cite[Theorem 9]{alo13} is not meaningful. One more mistake in this theorem is that $G_a^b$ is typed in stead of $G_c$. Even if it is corrected by replacing $G_a^b$ by $G_c$, as the convolution is not defined on $G_c\times G_c$, the statement is not meaningful.
\end{itemize}

In \cite{alo12}, one more attempt is made to extend the Hartley transform to the context of Boehmians on $L^p$-Boehmians. The first mistake is that Hartley transform is not defined on $L^p$, for an arbitrary $p>1$. So it is necessary to assume that $p=1$. As stated in the review Zbl 1261.46028, it is a replica of the definition given in \cite{lbd} and the proofs of theorems concerning the Hartley transform for integrable Boehmians are not proper and remain unexplained.

\end{document}